\documentclass{article}
\title{On certain properties of the Petty space}

\author{S.K.Mercourakis  \and G.Vassiliadis}

\date{}

\usepackage{amsmath}
\usepackage{amsfonts}
\usepackage{amssymb}
\usepackage{amscd}
\usepackage{amsthm}
\usepackage{makeidx}
\usepackage{euscript}
\usepackage{enumerate}
\usepackage{graphicx}
\usepackage{amsmath}
\usepackage{amssymb}
\usepackage{graphicx}

\theoremstyle{plain}
\newtheorem{theo}{Theorem}

\newtheorem*{Conjecture}{Conjecture}
\newtheorem{lemm}{Lemma}
\newtheorem{prop}{Proposition}
\newtheorem{cor}{Corollary}
\newtheorem{rem}{Remark}
\newtheorem{rems}{Remarks}
\theoremstyle{definition}

\newtheorem{Claim}{\underline{Claim}}
\newtheorem{example}{Example}
\newtheorem{examples}{Examples}
\newtheorem*{notation}{\underline{Notation}}

\begin{document}
\maketitle

\begin{abstract}
\footnotesize 
We study some touching properties of the three-dimensional Petty space 
$X=(\ell_2^2 \oplus \mathbb{R})_1$. In particular we give an estimation of its Hadwiger number
and also show that its equilateral subsets $A$ of maximum cardinality (i.e. $|A|=e(X)$) do not
have a center.


\normalsize 
\end{abstract}

\section*{Introduction}
Let $(X,||\cdot||)$ be a finite-dimensional normed space. The \textit{Hadwiger number} $H(X)$ of $X$ is
the maximum number of balls of $X$ of radius $1$ that can touch the unit ball $B$ of $X$, having pairwise 
disjoint interiors (the maximum number of balls of $X$ of radius $1$ kissing the unit ball of $X$), see \cite{Be},
\cite{Bo} and \cite{BMP}, also \cite{SWAN}. The \textit{strict Hadwiger number} $H'(X)$ of $X$ is the 
maximum number of pairwise disjoint balls of radius $1$ touching $B$. 

The \textit{equilateral number} $e(X)$ of $X$ is the largest cardinaliy of an equilateral subset of $X$, i.e. a subset
$S \subseteq X$ such that all pairwise distances between distinct points of $S$ are equall, see \cite{S}. 

A subset $S$ of $X$ will be called $1$-\textit{separated}, when $||x-y|| \ge 1$  $\forall x,y \in S$ with $x \neq y$.
$S$ will be called $1^+$-\textit{separated}, when $||x-y|| >1$  $\forall x,y \in S$ with $x \neq y$. It is easy to see
that $H(X)$ is the maximum cardinality of a 1-separated subset of the unit sphere $S_X$ of $X$ (resp. $H'(X)$ is the 
maximum cardinality of a $1^+$-separated subset of $S_X$).
A point $\alpha \in X$ is a center of $S$, if $\alpha$ is equidistant
from each point of $S$, i.e. there is $r>0$ such that $S \subseteq S(\alpha,r)=\{x \in S:||x-\alpha||=r\}$.

The 3-dimensional Petty space $X=(\ell_2^2 \oplus \mathbb{R})_1$ was introduced by C.M Petty, see \cite{P}. 
Let $\alpha=(x,y,z) \in X$; expressing $\alpha$ in cylindrical coordinates we write $\alpha=(r,\theta,z)$, 
where $r=\sqrt{x^2+y^2} \ge 0$ (when $r=0 \Leftrightarrow x=y=0$, then we set $r=\theta=0$). 
The norm of $\alpha$ is given by 

\begin{equation}
\begin{aligned}
||\alpha||=r+|z|.
\end{aligned}
\end{equation}

The distance of two points $\alpha_k=(r_k,\theta_k,z_k), k=1,2$ is
\begin{equation}
\begin{aligned}
||\alpha_1-\alpha_2||=\sqrt{r^2_1+r^2_2-2r_1r_2 \cos(\theta_1-\theta_2)}+|z_1-z_2|.
\end{aligned}
\end{equation}

The unit ball of this space is a double cone with cross-section of a square.
It is already known that $e(X)=5$ (see \cite{L}).

\section{Evaluation of the Hadwiger number of Petty space}

Let $S_X$ be the unit sphere of the 3-dimensional Petty space $X$. Given two points $\alpha_1 \neq \alpha_2 \in S_X \setminus \{\pm e_3\}$, the angle $\theta$ between the two vectors $(r_1,\theta_1)$ and $(r_2,\theta_2)$ of the polar plane, i.e. the angle
$\theta \in [0,\pi]$ such that $\cos \theta=\cos(\theta_1-\theta_2)$ will be called the \textit{angular distance}
of $\alpha_1,\alpha_2$.

Note that the angular distance of $\alpha_1,\alpha_2$ can be calculated using formula (2) as a function of
$z_1,z_2$ (or equivalently of the polar radii $r_k=1-|z_k|,k=1,2$) and of the distance $d=||\alpha_1-\alpha_2||$.
We denote by $\theta(z_1,z_2;d)$ the angular distance of $\alpha_1,\alpha_2$. In this paper we only consider the
1-\textit{angular distance} ($d=1$), which we denote by $\theta(z_1,z_2)$.

Regarding the 1-angular distance, by formula (2) we obtain:\\
\noindent (i) when $0 \le z_1 \le z_2 <1$
\begin{equation}
\begin{aligned}
\theta(z_1,z_2)=\arccos \left( 1-\frac{2(r_2-r_1)+1}{2r_1r_2} \right), 0<r_2 \le r_1
\end{aligned}
\end{equation}

(also when $-1<z_2 \le z_1 \le 0$ we find the same expression).

\noindent (ii) when $-1 < z_1 \le 0 \le z_2 <1$ and $z_2-z_1 \le 1$
\begin{equation}
\begin{aligned}
\theta(z_1,z_2)=\arccos \left( \frac{2(r_1+r_2)-1}{2r_1r_2}-1 \right).
\end{aligned}
\end{equation}

\begin{rems}

\begin{enumerate}
\item Given a pair of $z$-coordinates $z_1,z_2$ with $-1<z_1 \le z_2<1$ and $z_2-z_1 \le 1$,
formulas (3) and (4) determine the angle $\theta=\theta(z_1,z_2)$ such that when $\theta=\theta_1-\theta_2$,
then the points $\alpha_k=(r_k,\theta_k,z_k), k=1,2$ of $S_X \setminus \{\pm e_3\}$ have distance $||\alpha_1-\alpha_2||=1$,
i.e. they have 1-angular distance ($r_k=1-|z_k|,k=1,2$). 
\item Let $b_k=(r_k,\phi_k,z_k),k=1,2$ be two distinct points of $S_X \setminus \{\pm e_3\}$ with $|z_1-z_2| \le 1$,
$\theta=\theta(z_1,z_2)$ and $\phi$ be the angular distance of $b_1,b_2$. Then
\begin{equation}
\begin{aligned}
\phi \ge \theta \Leftrightarrow \cos \phi \le \cos \theta \Leftrightarrow ||b_1-b_2|| \ge 1.
\end{aligned}
\end{equation}

\item Let $S=\{\alpha_k:1 \le k \le n\}, n \ge 2$ be a finite subset of $S_X \setminus \{\pm e_3\}$, where
\begin{equation}
\begin{aligned}
\alpha_k=(r_k,\theta_k,z_k), -\frac{1}{2} \le z_k \le \frac{1}{2}, 1 \le k \le n \;\; \textrm{and} \;\; 
0 \le \theta_1 \le \theta_2 \le \dots \le \theta_n <2 \pi. 
\end{aligned}
\end{equation}
 
Suppose $S$ is 1-separated. Then summing the 1-angular distances of consecutive points of $S$ we have the inequality
\begin{equation}
\begin{aligned}
\sum_{k=1}^{n-1} \theta(z_k,z_{k+1})+\theta(z_n,z_1) \le 2 \pi.
\end{aligned}
\end{equation}

Indeed (taking into account (5)) we have:

$\theta(z_1,z_2)+\theta(z_2,z_3)+\dots+\theta(z_{n-1},z_n)+\theta(z_1,z_n) \le
(\theta_2-\theta_1)+(\theta_3-\theta_2)+\dots+(\theta_n-\theta_{n-1})+(\theta_1 +2 \pi-\theta_n)=2 \pi$.

It is also clear that (7) holds for any finite sequence $1 \le k_1<k_2<\dots<k_m \le n$, i.e.
\[\sum_{i=1}^{m-1} \theta(z_{k_i},z_{k_{i+1}})+\theta(z_{k_1},z_{k_m}) \le 2 \pi.  \]

Configuration (6) will be called the \textit{regular form} of $S$. Clearly, any finite subset of 
$S_X \setminus \{\pm e_3\}$ whose points have $z$-coordinates in $\left[ -\frac{1}{2},\frac{1}{2} \right]$ can be
arranged in regular form. Subsequently, whenever we consider a 1-separated subset of $S_X \setminus \{\pm e_3\}$
of this kind, we suppose that it is arranged in regular form.

\item Let $S=\{\alpha_k:1 \le k \le n\} \subseteq S_X \setminus \{\pm e_3\}$ in regular form.
In order to prove that $S$ is not 1-separated, it is sufficient (because of (7)) to find a sequence 
$\{\phi_k:1 \le k \le n \} \subseteq [0,\pi]$ with $\phi_k \le \theta(z_k,z_{k+1}),1 \le k \le n-1$, 
$\phi_n \le \theta(z_1,z_n)$ and 
\[\sum_{k=1}^n \phi_k>2 \pi.\]
\end{enumerate}
\end{rems}

We evaluate some 1-angular distances we are going to use in this section:

\begin{equation*}
\begin{aligned}
&\theta \left( \frac{1}{2},\frac{1}{3}\right)=\arccos0=90^o,\; &&\theta \left( \frac{1}{3},\frac{1}{4}\right)=
\arccos \left(\frac{1}{6} \right)>80.405^o, \\ 
&\theta \left( \frac{1}{3},-\frac{1}{3}\right)=
\arccos \left(\frac{7}{8} \right)>28.955^o, \; &&\theta \left( \frac{1}{3},-\frac{1}{4}\right)=
\theta \left( -\frac{1}{3},\frac{1}{4}\right)=\arccos \left(\frac{5}{6} \right)>33.557^o, \\ 
&\theta \left( -\frac{1}{4},\frac{1}{4}\right)=\arccos \left(\frac{7}{9} \right)>38.94^o, \;
&&\theta \left( \frac{1}{2},\frac{1}{4}\right)=\arccos \left(\frac{1}{3} \right)>70.528^o.
\end{aligned}
\end{equation*}

We first find a lower estimate for $H(X)$ and show that $H(X) \ge H'(X) \ge 14$.
This is established through the following:

\begin{example}
There is a $1^+$-separated subset of $S_X$ of 14 points. Let
\begin{equation*}
\begin{aligned}
&\alpha_1=(0.7 \cos0^o,0.7 \sin0^o,0.3),\; &&\alpha_2=(\cos45^o,\sin45^o,0), \\ 
&\alpha_3=(0.7 \cos90^o,0.7 \sin90^o,-0.3),\; &&\alpha_4=(\cos135^o,\sin135^o,0), \\
&\alpha_5=(0.7 \cos180^o,0.7 \sin180^o,0.3),\; &&\alpha_6=(\cos225^o,\sin225^o,0), \\
&\alpha_7=(0.7 \cos270^o,0.7 \sin270^o,-0.3),\; &&\alpha_8=(\cos315^o,\sin315^o,0).
\end{aligned}
\end{equation*}

The set $\{\alpha_i:1 \le i \le 8\}$ is easily checked to be $1^+$-separated because of the
1-angular distance $\theta(0.3,0) \approx 44.41^o<45^o$. Let also
\begin{equation*}
\begin{aligned}
&b_1=(0.501 \cos85^o,0.501 \sin85^o,0.499),\; &&b_2=(0.501 \cos265^o,0.501 \sin265^o,0.499), \\ 
&b_3=(0.501 \cos175^o,0.501 \sin175^o,-0.499),\; &&b_4=(0.501 \cos355^o,0.501 \sin355^o,-0.499).
\end{aligned}
\end{equation*}

One can check that the set $\{\alpha_i:1 \le i \le 8\} \cup \{b_j:1 \le j \le 4\} \cup \{\pm e_3\}$
is a $1^+$-separated subset of $S_X$ (see also \cite{MV}, Remarks 7(2)).
\end{example}

To obtain an upper bound of $H(X)$ we will use the following Lemma, which states that when the vertical
distance of two points decreases, their 1-angular distance increases.

\begin{lemm}
Let $-\frac{1}{2} \le z'_1 \le z_1 \le z_2 \le z'_2 \le \frac{1}{2}$. Then $\theta(z_1,z_2) \ge \theta(z'_1,z'_2)$.
\end{lemm}

\begin{proof}
We first consider the case when all the points lie on the same hemisphere. Suppose that 
$0 \le z'_1 \le z_1 \le z_2 \le z'_2 \le \frac{1}{2}$ (the case $-\frac{1}{2} \le z'_2 \le z_2 \le z_1 \le z'_1 \le 0$
is similar, due to symmetry).

Fix $z_2$ and let $0 \le z \le z_2 \le \frac{1}{2}$, $r=1-|z|$ and $r_2=1-|z_2|$, $\frac{1}{2} \le r_2 \le r$.
We differentiate the expression of $\cos \theta(z,z_2)$ from (3) with respect to $r$:
\[\left( 1-\frac{2(r_2-r)+1}{2r_2r} \right)'=\frac{4 r^2_2+2r_2}{4r^2_2r^2}>0 \]

so as $z$ increases to $z_2$ (then $r$ decreases to $r_2$), the 1-angular distance $\theta(z,z_2)$ increases.

Now fix $z_1$ and let $0 \le z_1 \le z \le \frac{1}{2}$, $r=1-|z|$ and $r_1=1-|z_1|$, $\frac{1}{2} \le r \le r_1$.
We differentiate the expression of $\cos \theta(z,z_1)$ from (3) with respect to $r$:
\[\left( 1-\frac{2(r-r_1)+1}{2r_1r} \right)'=\frac{2 r_1(1-2r_1)}{4r^2_1r^2} \le 0 \]

so as $z$ decreases to $z_1$ (then $r$ increases to $r_1$), the 1-angular distance $\theta(z,z_1)$ increases
and the first case is valid.

For the general case, suppose first that $-\frac{1}{2} \le z_1 \le 0 \le z_2 \le \frac{1}{2}$. 
Fix $z_1$ and let $-\frac{1}{2} \le z_1 \le 0 \le z \le \frac{1}{2}$, $r=1-|z|$ and $r_1=1-|z_1|$, $\frac{1}{2} \le r , r_1$.
We differentiate the expression of $\cos \theta(z,z_1)$ from (4) with respect to $r$:
\[\left( \frac{2(r+r_1)-1}{2r_1r}-1 \right)'=\frac{2 r_1(1-2r_1)}{4r^2_1r^2} \le 0 \]

so as $z$ decreases to $0$ (then $r$ increases to $1$), the 1-angular distance $\theta(z,z_1)$ increases.
The case $-\frac{1}{2} \le z \le 0 \le z_2 \le \frac{1}{2}$ is similar due to symmetry.

The remaining case is when $-\frac{1}{2} \le z'_1 \le z_1 \le z_2 \le 0 \le z'_2 \le \frac{1}{2}$.
By the previous reasoning we have that
\[\theta(z'_1,z'_2) \le \theta(z'_1,0) \le \theta(z_1,z_2)\]
where in the first inequality $z'_1,0$ are considered to lie on different hemispheres, while
in the second we consider them on the same hemisphere.
\end{proof}

We partition the unit sphere $S_X$ into the following subsets:
\begin{equation*}
\begin{aligned}
&D_1=\left\{(x,y,z) \in S_X:z>\frac{1}{2} \right\},\; &&D'_1=\left\{(x,y,z) \in S_X:z<-\frac{1}{2}\right\}, \\ 
&D_2=\left\{(x,y,z) \in S_X:\frac{1}{3}<z \le \frac{1}{2}\right\},\; 
&&D'_2=\left\{(x,y,z) \in S_X:-\frac{1}{2} \le z <-\frac{1}{3}\right\}, \\
&D_0=\left\{(x,y,z) \in S_X:-\frac{1}{3} \le z \le \frac{1}{3}\right\}.
\end{aligned}
\end{equation*}

Clearly $S_X=D_0 \cup D_1 \cup D'_1 \cup D_2 \cup D'_2$ is a disjoint union. We are going to evaluate
the maximum cardinality of a 1-separated subset of each of these sets.

\begin{prop}

\begin{enumerate}[(i)]
\item A 1-separated subset of $D_1$ (or $D'_1$ due to symmetry) has cardinality at most 1.
\item A 1-separated subset of $D_2$ (or $D'_2$ due to symmetry) has cardinality at most 3.
\end{enumerate}
\end{prop}

\begin{proof}
\noindent (i) Let $a=(r_1,0,1-r_1)$ and $b=(r_2 \cos \theta,r_2 \sin \theta,1-r_2)$
with $a,b \in D_1$ (note that the rotation around the $z$-axis is an isometry, so we may
suppose that one of the points has cylindrical coordinate $\theta=0$). Then
\[||a-b||=\sqrt{(r_1+r_2)^2-2r_1r_2(1+\cos \theta)}+|r_1-r_2| \le r_1+r_2+|r_1-r_2|<1\]
in any case, since $r_1,r_2<\frac{1}{2}$ for $a,b \in D_1$.

\noindent (ii) Let $\alpha_i=(x_i,y_i,z_i) \in D_2,i=1,2$. Then $\frac{1}{3} <z_1,z_2 \le \frac{1}{2}$,
so from Lemma 1 we get that 
\[\theta(z_1,z_2)>\theta \left( \frac{1}{2},\frac{1}{3} \right)=90^o.\]

Now for $S \subseteq D_2$ in regular form with $|S|=4$, the sum of the 1-angular distances of consecutive
points is $>2 \pi$, hence $S$ cannot be 1-separated (see also Remarks 1(3) and (4)).
\end{proof}

We next prove that a 1-separated subset of $D_0$ can have at most 10 points. Let
\begin{equation*}
\begin{aligned} 
&D_3=\left\{(x,y,z) \in S_X:\frac{1}{4}<z \le \frac{1}{3}\right\},\; 
&&D'_3=\left\{(x,y,z) \in S_X:-\frac{1}{3} \le z <-\frac{1}{4}\right\} \; \textrm{and} \\
&\overline{D_0}=\left\{(x,y,z) \in S_X:-\frac{1}{4} \le z \le \frac{1}{4}\right\}.
\end{aligned}
\end{equation*}

\begin{prop}
A 1-separated subset of $D_0$ has cardinality at most 10.
\end{prop}

\begin{proof}
Let $S \subseteq D_0$ be 1-separated with $|S|=11$. Since $\theta \left( \frac{1}{4},\frac{1}{3} \right)>80.405^o$
we get that $|S \cap D_3| \le 4$ and $|S \cap D'_3| \le 4$ (see also Prop.1(ii)).

\begin{Claim}
For $S \subseteq D_0$ 1-separated with $|S|=11$ we must have that $|S \cap D_3|=4$ and $|S \cap D'_3| \ge 3$
(or vice versa).
\end{Claim}
 \begin{proof} (of the Claim)
 Suppose that $|S \cap D_3|=|S \cap D'_3|=3$. The remaining 5 points of $S$ lie on $S \cap \overline{D_0}$.
 
 Note that assuming two consecutive points of $S \cap D_3$ (or $S \cap D'_3$), applying Lemma 1 
 gives a sum of 1-angular distances
 
 \[ \ge \theta \left(\frac{1}{3},\frac{1}{4} \right)+10 \cdot \theta \left(-\frac{1}{3},\frac{1}{3} \right)
 \ge 80.405^o+10 \cdot 28.955^o=369.955^o \]
 so this is not possible.
 
 We sum the 1-angular distances of consecutive points of $S$ (see Remarks 1(3) and (4)) examining the cases
 of how many pairs of consecutive points, one from $D_3$ and one from $D'_3$ we have (there can be up to 5 such pairs).
 
 If there is no such pair, then each 1-angular distance of consecutive points is 
 $\ge \theta \left(\frac{1}{3},-\frac{1}{4} \right) >33.557^o$ and the sum is $>11 \cdot 33.557^o=369.127^o$.
 So there can be 1 to 5 such pairs of points. Observing the regular form of $S$, any point of $D_3$ can either
 have 2 neighbouring points from $D'_3$, or one from $D'_3$ and one from $\overline{D_0}$, 
 or two points from $\overline{D_0}$.
 
 Examining cases for $k=1,2,\dots,5$ the number of pairs of consecutive points, one from $D_3$ and one from $D'_3$,
 we find a sum of 1-angular distances
 \[ \ge k \cdot \theta \left(\frac{1}{3},-\frac{1}{3} \right)+2(6-k) \cdot \theta \left(\frac{1}{3},-\frac{1}{4} \right)+
 (k-1) \cdot \theta \left(-\frac{1}{4},\frac{1}{4} \right)\]
 \[>k \cdot 28.955^o+2(6-k) \cdot 33.557^o +(k-1) \cdot 38.94^o>360^o\]
 for $k=1,2,\dots,5$. So we cannot have both $|S \cap D_3|<4$ and $|S \cap D'_3|<4$.
 
 Suppose now that $|S \cap D_3|=4$ and $|S \cap D'_3|=2$. If there is no pair of consecutive points from $D_3$ and $D'_3$,
 as in the previous case we obtain a contradiction. In this case there can be up to 4 such pairs of points and exactly the
 same calculation as before shows that the sum of 1-angular distances is $>360^o$.
 
 Assuming that less than 6 points of $S$ lie on $D_3 \cup D'_3$, by analogous calculations we get even larger sums
 of 1-angular distances and the proof of the Claim is complete. 
 \end{proof}

So we may assume that $|S \cap D_3|=4$, $|S \cap D'_3| \ge 3$ and $|S \cap \overline{D_0}| \ge 3$.
We consider the regular form of $S$ and observe the relative position of points of $S \cap D_3$ with respect
to other points of $S$.

Suppose we have two pairs of points of $S \cap D_3$ with a single other point of $S$ between them. The sum
of 1-angular distances is then
\[\ge 2 \cdot \theta \left(\frac{1}{3},\frac{1}{4} \right)+7 \cdot \theta \left(-\frac{1}{3},\frac{1}{3} \right)
>2 \cdot 80.405^o+7 \cdot 28.955^o =363.495^o.\]

Recall that it is not possible to have two consecutive points of $S \cap D_3$ or of $S \cap D'_3$. Hence exactly one pair of
points of $S \cap D_3$ has a single other point of $S$ between them and 3 pairs of points of $S \cap D_3$ have two other points 
of $S$ between them, which results to a sum of 1-angular distances
\[\ge 3 \left[ \theta \left(-\frac{1}{3},\frac{1}{3} \right)+2 \cdot \theta \left(-\frac{1}{3},\frac{1}{4} \right) \right]+
\theta \left(\frac{1}{3},\frac{1}{4} \right) >3 (28.955^o+2 \cdot 33.557^o)+80.405^o =368.612^o\]
which is also impossible and the Proposition holds. 

\end{proof}

Combining the results of Propositions 1 and 2 we obtain the following easy

\begin{cor}
$H(X) \le 18$.
\end{cor}

\begin{proof}
$S_X=D_0 \cup D_1 \cup D'_1 \cup D_2 \cup D'_2$ (a disjoint union).
\end{proof}

\begin{lemm}
Let $S \subseteq S_X$ be 1-separated.

\begin{enumerate}[(i)]
\item When $|S \cap D_3|=4$, then $S \cap D_2=\emptyset$.
\item When $|S \cap D_3|=3$, then $|S \cap D_2| \le 1$.
\end{enumerate}
Due to symmetry, an analogous result holds for $S \cap D'_3$ and $S \cap D'_2$.
\end{lemm}

\begin{proof}
We take into account that $\theta \left(\frac{1}{3},\frac{1}{4} \right) >80.405^o$ and 
$\theta \left(\frac{1}{2},\frac{1}{4} \right) >70.528^o$. When $|S \cap D_3|=4$,
assuming $S \cap D_2 \neq \emptyset$ we place the points of $S \cap (D_3 \cup D_2)$ 
around the circle (see also Remarks 1(3) and (4)) and one can check that we have a sum of 1-angular distances
\[\ge 2 \cdot \theta \left(\frac{1}{4},\frac{1}{2} \right)+3 \cdot \theta \left(\frac{1}{3},\frac{1}{4} \right)
>2 \cdot 70.528^o+3 \cdot 80.405^o =382.271^o\]
(because of Lemma 1) so (i) holds.

The proof of (ii) is similar, paying attention to the fact that two points of $S \cap D_2$  must have 1-angular distance
at least $\theta \left(\frac{1}{2},\frac{1}{3} \right)=90^o$. 
\end{proof}

Let now 
\begin{equation*}
\begin{aligned} 
&D_4=\left\{(x,y,z) \in S_X:0.245<z \le \frac{1}{3}\right\} \; \textrm{and}
&&D'_4=\left\{(x,y,z) \in S_X:-\frac{1}{3} \le z <-0.245 \right\}.
\end{aligned}
\end{equation*}

\begin{lemm}
Let $S \subseteq S_X$ be 1-separated.
\begin{enumerate}[(i)]
\item When $|S \cap D_4| \ge 2$, then $|S \cap D_2| \le 2$ (respectively when $|S \cap D'_4| \ge 2$, 
then $|S \cap D'_2| \le 2$).
\item When $|S \cap D_4| \ge 3$ and $|S \cap D_3|=2$, then $|S \cap D_2| \le 1$
(respectively when $|S \cap D'_4| \ge 3$ and $|S \cap D'_3|=2$, then $|S \cap D'_2| \le 1$).
\end{enumerate}

\end{lemm}

\begin{proof} 
Note that two points of $D_2$ have 1-angular distance $>90^o$ and two points of $D_4$ have 1-angular distance 
$ \ge \theta \left(\frac{1}{3},0.245 \right)>79.5^o$. A point of $D_4$ and a point of $D_2$ have 1-angular distance 
$ \ge \theta \left(\frac{1}{2},0.245 \right)>69.45^o$. Also a point of $D_4$ and a point of $D_3$ have 1-angular distance 
$ \ge \theta \left(\frac{1}{3},0.245 \right)>79.5^o$. 

If $S$ has two points in $D_4$ and 3 points in $D_2$, the sum of 1-angular distances is $>360^o$ in any case
(the reader may draw a sketch of the points around the circle; there are two cases, to have one pair of consecutive points
of $D_2$ or to have two pairs of this kind) and (i) of Lemma 3 holds.

The reasoning for (ii) is similar. If we place around the circle two points of $D_3$, one point of $D_4$ and two points
of $D_2$ and sum the 1-angular distances, in any case the sum is $>360^o$. 
\end{proof}

\begin{prop}
Let $S \subseteq S_X$ be 1-separated. If $|S \cap D_0|=10$, then $|S| \le 16$. 
\end{prop}

\begin{proof}
When $|S \cap D_3|=4$ or $|S \cap D'_3|=4$, then by Lemma 2 we have $S \cap D_2=\emptyset$ or $S \cap D'_2=\emptyset$
and as in Corollary 1 we get $|S| \le 10+2+3=15$.

When $|S \cap D_3|=3$ or $|S \cap D'_3|=3$, by Lemma 2 we have $|S \cap D_2|=1$ or $|S \cap D'_2|=1$
and as in Corollary 1 we get $|S| \le 10+2+1+3=16$.

\begin{Claim}
For $S \subseteq S_X$ 1-separated with $|S \cap D_0|=10$ such that $|S \cap D_3|<3$ and $|S \cap D'_3|<3$, 
we must have $|S \cap D_3|=2$ and $|S \cap D'_3|\ge 1$ or vice versa.

Also one of the following is valid
\begin{enumerate}[(i)]
\item $|S \cap D_4| \ge 2$ and $|S \cap D'_4|\ge 2$ or
\item if $|S \cap D'_4| \le 1$, then $|S \cap D_4| \ge 3$ and $|S \cap D_3|=2$ or
\item if $|S \cap D_4| \le 1$, then $|S \cap D'_4| \ge 3$ and $|S \cap D'_3|=2$.
\end{enumerate}
\end{Claim}

\begin{proof} (of Claim 2)
Suppose first that $|S \cap D_3|=|S \cap D'_3|=1$. Examining the cases for $k=0,1$ the number
of pairs of consecutive points, one from $S \cap D_3$ and one from $S \cap D'_3$, for the points of $S \cap D_0$
we find a sum of 1-angular distances
\[ \ge k \cdot \theta \left(\frac{1}{3},-\frac{1}{3} \right)+2(2-k) \cdot \theta \left(\frac{1}{3},-\frac{1}{4} \right)+
(6+k) \cdot \theta \left(-\frac{1}{4},\frac{1}{4} \right)\]
\[>k \cdot 28.955^o+2(2-k) \cdot 33.557^o +(6+k) \cdot 38.94^o>360^o\]
for $k=0,1$, so this is impossible.
 
Also if $|S \cap D_3|=2$ and $S \cap D'_3=\emptyset$ (or vice versa), we have a sum of 1-angular distances
\[ \ge 4 \cdot \theta \left(-\frac{1}{4},\frac{1}{3} \right)+6 \cdot \theta \left(\frac{1}{4},-\frac{1}{4} \right)
>4 \cdot 33.557^o +6 \cdot 38.94^o=367.868^o.\] 

Supposing that $|S \cap D_3|=1$ and $S \cap D'_3=\emptyset$ (or vice versa) would result to an even larger
sum of 1-angular distances, hence the first part of the Claim holds.

For the second part, let $|S \cap D_4|=2$ and $|S \cap D'_4|=1$ (or vice versa). Since 
\[\theta \left(-\frac{1}{3},0.245 \right)>33.78^o \; \textrm{and} \;\theta \left(-0.245,0.245 \right)>39.47^o,\]
the sum of 1-angular distances of $S \cap D_0$ is
\[ \ge k \cdot \theta \left(\frac{1}{3},-\frac{1}{3} \right)+2(3-k) \cdot \theta \left(-\frac{1}{3},0.245 \right)+
(4+k) \cdot \theta \left(-0.245,0.245 \right)\]
\[>k \cdot 28.955^o+2(3-k) \cdot 33.78^o +(4+k) \cdot 39.47^o>360^o\]
for $k=0,1,2$ the number of pairs of consecutive points, one from $S \cap D_3$ and one from $S \cap D'_3$.
So the second part of the Claim is also valid. 
\end{proof}

When (i) of Claim 2 is valid, because of Lemma 3(i) we can have at most two points in $S \cap D_2$ and two points in $S \cap D'_2$,
so as in Corollary 1 we have $|S| \le 10+2+2+2=16$. In the other two cases of Claim 2, because of Lemma 3(ii) we have $|S| \le 10+2+3+1=16$. It is now clear that Proposition 3 is valid.
\end{proof}

In order to resolve the case $|S \cap D_0|=9$ we need the following:

\begin{lemm}
Let $\frac{1}{2}>z_1\ge 0\ge z_2>-\frac{1}{2}$, $r_1=1-z_1$ and $r_2=1+z_2$. For any 
$z \in \left[-\frac{1}{2},\frac{1}{2}\right]$ with $z_1 \ge z \ge z_2$ the sum of 1-angular distances
$\theta(z_1,z)+\theta(z,z_2)$ is minimum when $z=0$. That is, for 3 points with decreasing (or increasing)
with respect to their $z$-coordinate order, the sum of consecutive 1-angular distances is minimum when
the middle point has $z=0$.
\end{lemm}

\begin{proof}
Since the reflection with respect to the $xy$-plane is an isometry of $X$, we may suppose that $z \ge 0$ and $r=1-z$.
From (3) and (4) we have 
\[\theta(z_1,z)+\theta(z,z_2)=\arccos \left( 1-\frac{2(r_1-r)+1}{2r_1r} \right)+
\arccos \left( \frac{2(r+r_2)-1}{2r_2r}-1 \right), \; \frac{1}{2}<r \le 1.\]
Let
\[\alpha(r)=1-\frac{2(r_1-r)+1}{2r_1r} \Rightarrow \alpha'(r)=\frac{2r_1+1}{2r_1r^2}>0\]
and let 
\[\beta(r)=\frac{2(r+r_2)-1}{2r_2r}-1 \Rightarrow \beta'(r)=\frac{1-2r_2}{2r_2r^2}<0 \; \;
\left(\frac{1}{2}<r_2 \le 1 \right).\]
Let now 
\[f(r)=\arccos(\alpha(r))+\arccos (\beta(r)) \Rightarrow\]
\[f'(r)=-\frac{\alpha'(r)}{\sqrt{1-\alpha^2(r)}}-\frac{\beta'(r)}{\sqrt{1-\beta^2(r)}}=
-\frac{\alpha'(r)\sqrt{1-\beta^2(r)}+\beta'(r)\sqrt{1-\alpha^2(r)}}{\sqrt{1-\alpha^2(r)}\sqrt{1-\beta^2(r)}}.\]
We will show that $f'(r)<0$, $\frac{1}{2}<r \le 1$, hence $f(r)$ is minimum for $r=1 \Leftrightarrow z=0$.
It suffices to show that 
\[\alpha'(r)\sqrt{1-\beta^2(r)}+\beta'(r)\sqrt{1-\alpha^2(r)}>0.\]
If we replace the values of $\alpha(r),\beta(r),\alpha'(r),\beta'(r)$ and do some operations, 
we reach an inequality of the form
\[Ar^2+Br>0\]
where $A=8(2r_1+1)(2r_2-1)(r_1+r_2)$ and $B=-\frac{A}{2}$, so 
\[Ar^2-\frac{A}{2}r>0\]
which holds for $r>\frac{1}{2}$ and hence $f(r)$ is a decreasing function.
\end{proof}

Let now 
\begin{equation*}
\begin{aligned} 
&D_5=\left\{(x,y,z) \in S_X:0.224<z \le \frac{1}{3}\right\} \; \textrm{and}
&&D'_5=\left\{(x,y,z) \in S_X:-\frac{1}{3} \le z <-0.224 \right\}.
\end{aligned}
\end{equation*}

\begin{lemm}
Let $S \subseteq S_X$ be 1-separated.
When $|S \cap D_3| \ge 1$ and $|S \cap D_5| \ge 2$, then $|S \cap D_2| \le 2$ (repsectively when $|S \cap D'_3| \ge 1$ 
and $|S \cap D'_5| \ge 2$, then $|S \cap D'_2| \le 2$).
\end{lemm}

\begin{proof} 
A point of $D_2$ and a point of $D_5$ have 1-angular distance $ \ge \theta \left(\frac{1}{2},0.224 \right)>64.99^o$. Also a point
of $D_3$ and a point of $D_5$ have 1-angular distance $ \ge \theta \left(\frac{1}{3},0.224 \right)>75.82^o$. Recall that the 1-angular distance between a point of $D_3$ and a point of $D_2$ is $ \ge \theta \left(\frac{1}{2},\frac{1}{4} \right)>70.528^o$. 

If we place around the circle one point of $D_3$, one point of $D_5$ and 3 points of $D_2$ and sum the 1-angular distances, we
see that in any case the sum is $>360^o$. So the proof of the Lemma is complete. 
\end{proof}

\begin{prop}
Let $S \subseteq S_X$ be 1-separated. If $|S \cap D_0|=9$, then $|S| \le 16$.
\end{prop}

\begin{proof}
When $|S \cap D_3| \ge 2$, then also $|S \cap D_4| \ge 2$, so $|S \cap D_2| \le 2$ (see Lemma 3(i)).
The same holds for $S \cap D'_3, S \cap D'_4$ and $S \cap D'_2$. So when $|S \cap D_3| \ge 2$ or $|S \cap D'_3| \ge 2$,
as in Corollary 1 we get $|S| \le 9+2+3+2=16$.

\begin{Claim}
For $S \subseteq S_X$ 1-separated with $|S \cap D_0|=9$ one of the following holds:
\begin{enumerate}[(i)]
\item $|S \cap D_3| \ge 1$ and $|S \cap D_5| \ge 2$ or 
\item $|S \cap D'_3| \ge 1$ and $|S \cap D'_5| \ge 2$
\end{enumerate} 
\end{Claim}

\begin{proof} (of Claim 3)
Suppose that $|S \cap D_3|=|S \cap D'_3|=1$ and that the rest of the points of $S \cap D_0$ have $z$-coordinate
in $[-0.224,0.224]$. These last points have 1-angular distance from one another $\ge \theta \left(-0.224,0.224 \right)>41.66^o$,
from points of $D_3$ (or $D_5$) $\ge \theta \left(-\frac{1}{3},0.224 \right)>34.69^o$ and from points with zero $z$-coordinate
1-angular distance $\ge \theta \left(0.224,0 \right)>49.88^o$. We will also use the fact that $\theta \left(\frac{1}{3},0 \right)>41.4096^o$.

When $|S \cap D_0|=9$, observing its regular form we see that around the circle there is at least one triplet of consecutive 
points in increasing or decreasing (not strictly) with respect to their $z$-coordinate order. According to Lemma 4 the minimum sum
of 1-angular distances in such a triplet occurs when the middle point has $z$-coordinate $=0$. Under our assumption we have 
two cases.

The first one is when the points of $S \cap D_3$ and $S \cap D'_3$ are not consecutive. In this case the sum of 1-angular 
distances is 
\[ \ge k \cdot \theta \left(\frac{1}{3},0 \right)+(2-k) \cdot \theta \left(0.224,0 \right)+(4-k) \cdot 
\theta \left(-\frac{1}{3},0.224 \right)+(k+3) \cdot \theta (0.224,-0.224)\]
\[>k \cdot 41.4096^o+(2-k) \cdot 49.88^o +(4-k) \cdot 34.69^o+(k+3) \cdot 41.66^o>360^o\]
for $k=0,1,2$ the number of points of $S \cap (D_3 \cup D'_3)$ which are first or last in the triplet.

The second case is when the points of $S \cap D_3$ and $S \cap D'_3$ are consecutive. Then we have a sum of 1-angular distances
\[ \ge \theta \left(\frac{1}{3},-\frac{1}{3} \right)+ k \cdot \theta \left(\frac{1}{3},0 \right)+(2-k) \cdot 
\theta \left(0.224,0 \right)+(2-k) \cdot \theta \left(-\frac{1}{3},0.224 \right)+(4+k) \cdot \theta (0.224,-0.224)\]
\[>28.955+k \cdot 41.4096^o+(2-k) \cdot 49.88^o +(2-k) \cdot 34.69^o+(4+k) \cdot 41.66^o>360^o\]
for $k=0,1$ the number of points of $S \cap (D_3 \cup D'_3)$ which are first or last in the triplet.

Clearly, supposing that $|S \cap (D_3 \cup D'_3)| \le 1$ would result to even larger sums of 1-angular distances and the proof 
of Claim 3 is complete.
\end{proof}

Because of Claim 3 and Lemma 5, as in Corollary 1 we have that $|S| \le 9+2+3+2=16$. 

\end{proof}

\begin{theo}
For the Hadwiger number of Petty space the following bounds are established
\[14 \le H'(X)\le H(X) \le 16.\]
\end{theo}

\begin{proof}
The lower bound was shown in Example 1. Recall the partition of the unit sphere $S_X$ into certain zones (see Corollary 1). 
If $S \subseteq S_X$ is 1-separated, from Prop.1 we have that $|S \cap D_1| \le 1$ and $|S \cap D'_1| \le 1$;
also $|S \cap D_2| \le 3$ and $|S \cap D'_2| \le 3$. 

Proposition 2 implies that $|S \cap D_0| \le 10$ and from Propositions 3,4 we have that when $|S \cap D_0|=9$ or $10$,
then $|S| \le 16$. Also if $|S \cap D_0| \le 8$, by the partition of Corollary 1 and Proposition 1 we get that
$|S| \le 8+2+3+3=16$ and the upper bound is also established.
\end{proof}

We believe that a refinement of our method can yield $H(X) \le 15$ but the presentation of the details of such an 
analysis would be significantly longer. Regarding the Hadwiger number of Petty space the following is likely to hold

\begin{Conjecture}
Prove that $H(X)=H'(X)=14$. 
\end{Conjecture}

\section{A space whose equilateral sets of maximum cardinality do not have a center}
Let $X$ be the 3-dimensional Petty space with $||(x,y,z)||=\sqrt{x^2+y^2}+|z|$. Recall that the maximum cardinality
of an equilateral set in $X$ is $e(X)=5$.
There is an infinite number of configurations of 5 points yielding an equilateral set
(equivalently of 5 double cone balls with the same radius touching each other), see also
Remark 3.

\begin{notation}
We denote by $A=\{\alpha_i:i=1,2,\dots,5\}$ a 5-point 1-equilateral set in $X$.
We also denote by $\alpha'_i, i=1,2,\dots,5$ the $z$-coordinate of the corresponding point
and we suppose that $\alpha'_5 \le \alpha'_3 \le \alpha'_2 \le \alpha'_1 \le \alpha'_4$.
Since the translation and the rotation around the $z$-axis are isometries, 
we may always consider that $\alpha_5=(0,0,0)$ and $\alpha_4=(0,d,1-d)$, 
where $d$ is the "submersion" of the $\alpha_4$-centered
cone with respect to the highest possible $z$-coordinate position $\alpha'_4=1$. 
\end{notation}

\begin{figure}[htbp]
\centering
\includegraphics[width=9cm, height=4cm]{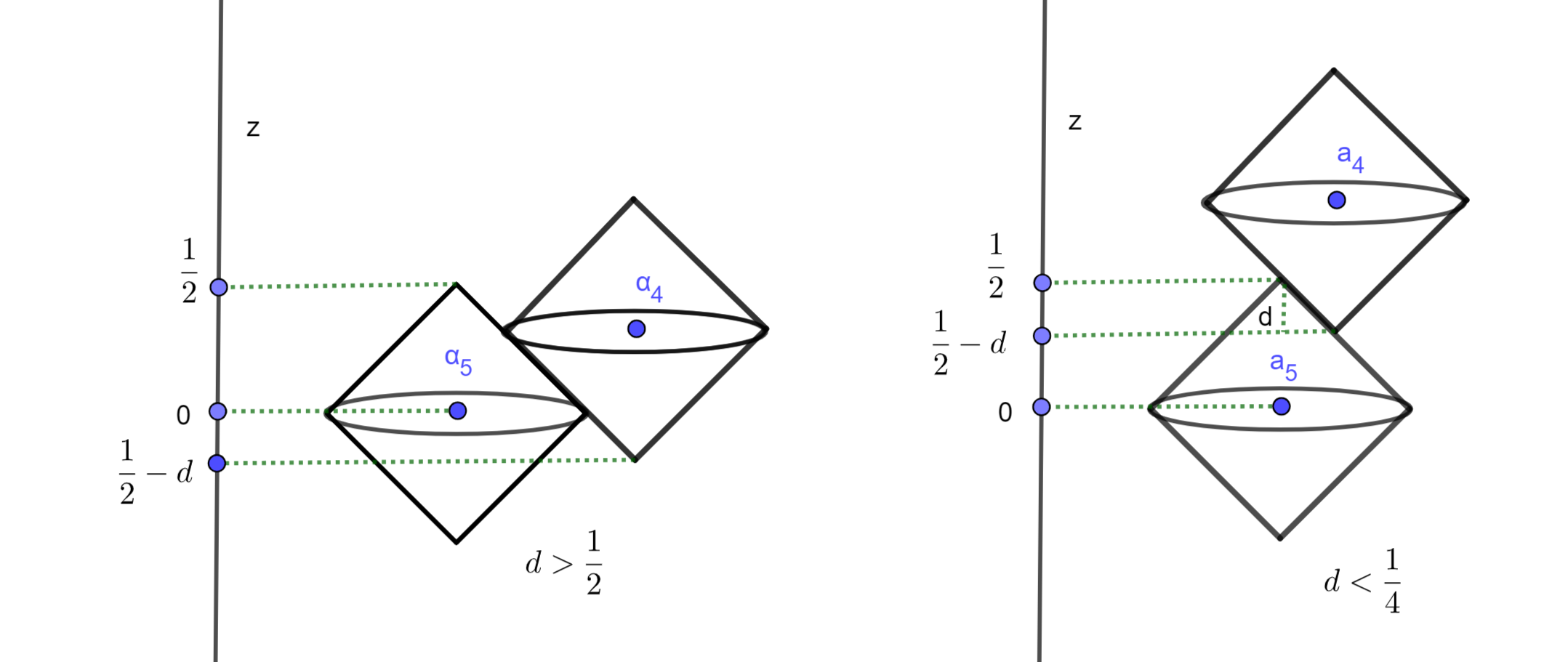}
\caption{}
\label{fig:figure 1}
\end{figure}

\begin{prop}
Let $A=\{\alpha_i:i=1,2,\dots,5\}$ be a 1-equilateral subset of $X$ 
with $\alpha_5=(0,0,0)$ and $\alpha_4=(0,d,1-d)$. Then necessarily $d<\frac{1}{4}$
and the remaining points $\alpha_i, i=1,2,3$ lie on a planar curve (ellipse).
Also $\frac{1}{2}-d \le \alpha'_i \le \frac{1}{2}, i=1,2,3$.
\end{prop}

\begin{proof}
When $d>\frac{1}{2}$ it is obvious that there are at most 4 cones with radius $\frac{1}{2}$
touching each other (equivalently a maximal 1-equilateral set of 4 points), see Fig.1.

Let $0<d \le \frac{1}{2}$. The fact that $\frac{1}{2}-d \le \alpha'_i \le \frac{1}{2}, i=1,2,3$
is easy to check, because each cone with center $\alpha_i, i=1,2,3$ must touch both of the other cones.

Consider a point $\alpha=(x,y,z) \in X$ with $||\alpha-\alpha_4||=||\alpha-\alpha_5||=1$.
The following equations hold (note that $0 \le \frac{1}{2}-d \le z \le \frac{1}{2}$):

\begin{equation}
\left.
\begin{aligned}
\sqrt{x^2+y^2}+z=1\\
\sqrt{x^2+(y-d)^2}+(1-d)-z=1
\end{aligned}
\right \}
\end{equation}

and solving for $z$ we get

\begin{equation}
\left.
\begin{aligned}
z=ay+b\\
x^2=\frac{2d+1}{4(d+1)^2} [-4y^2+4dy+2d+1]
\end{aligned}
\right \}
\end{equation}

where $a=-\frac{d}{d+1}$ and $b=\frac{1}{2(1+d)}$. One can check that the graph of (9)
is an ellipse lying on the plane $z=ay+b$. The points $\alpha_i, i=1,2,3$ lie on this ellipse
and are the centers of double cones of radius $\frac{1}{2}$ touching each other and also
touching the cones with centers $\alpha_4,\alpha_5$. It is useful to have in mind the arrangement
of the points $\{\alpha_i,i=1,2,\dots,5\}$ in space.

When $d=\frac{1}{4}$, it is easy to check that the 4 vertices of the corresponding
ellipse $A_1=\left(0,-\frac{1}{2},\frac{1}{2} \right)$, 
$A_2=\left(-\sqrt{\frac{3}{8}},\frac{1}{8},\frac{3}{8} \right)$, 
$A_3=\left(0,\frac{3}{4},\frac{1}{4} \right)$ and
$A_4=\left(\sqrt{\frac{3}{8}},\frac{1}{8},\frac{3}{8} \right)$ form a rhombus of 
side length 1. The 4 vertices partition the ellipse into 4 quadrants. If we fix 
a point $A$ in a certain quadrant, there are exactly two points $B,C$ on the ellipse lying in 
quadrants neighbouring the initial one and which are at distance 1 from the first point 
(equivalently one fixes a cone whose center lies on the ellipse and sets two other cones moving 
their centers around the ellipse until they touch the fixed cone). The two points $B,C$
defined this way are at distance $>1$ because they do not belong to neighbouring quadrants (see also Fig. 2).

When $\frac{1}{4} < d \le \frac{1}{2}$ the vertices of the ellipse are now at distance $>1$ and as in the previous case
one cannot have a 1-equilateral triangle whose vertices lie on the corresponding ellipse (equivalently one cannot
have 3 cones with radii $\frac{1}{2}$, whose centers lie on the ellipse and so that they touch each other).
\end{proof}

\begin{figure}[htbp]
\centering
\includegraphics[width=7cm, height=4cm]{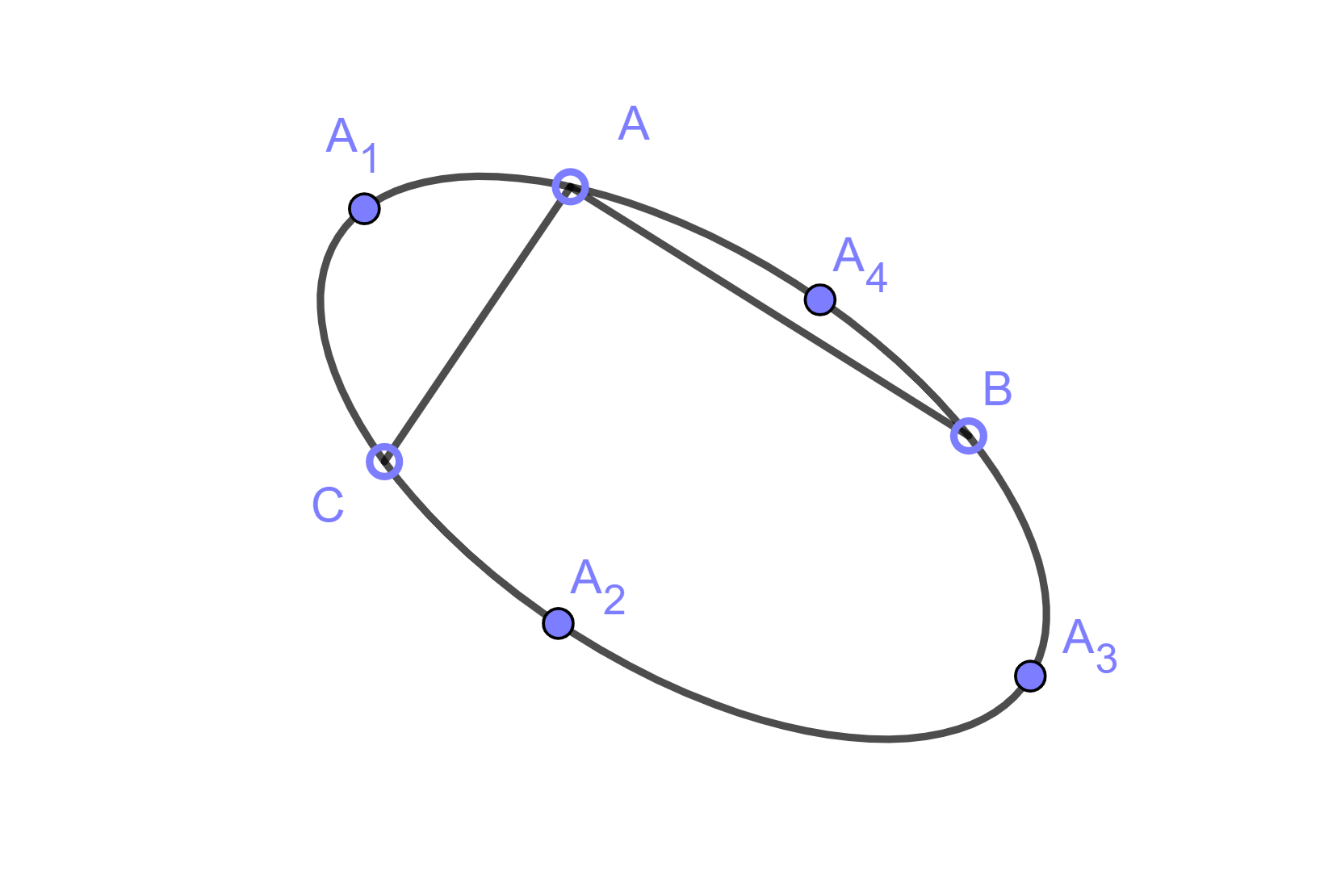}
\caption{}
\label{fig:figure 2}
\end{figure}

\begin{rem}
When the dimension of a space is $\ge 3$, Petty's Theorem (see [Pe]) states that fixing a 3-point equilateral set
we can always find a fourth point so that the 4 points form an equilateral set. In case when $\frac{1}{4} \le d \le 1$
Petty's Theorem and Proposition 1 ensure the existence of exactly 4 cones touching each other (equivalently of a maximal
equilateral set of 4 points).  
\end{rem}

We present two interesting examples of 4-point maximal equilateral sets in $X$, which
refer to the special case $d=1$.

\begin{examples}
\begin{enumerate}
\item Let $A_1=\left(0,\frac{1}{2},0\right)$, $A_2=\left(0,-\frac{1}{2},0\right)$, 
$A_3=\frac{1}{2}\left(1,0,2-\sqrt{2}\right)$, $A_4=\frac{1}{2}\left(-1,0,2-\sqrt{2}\right)$ 
and $K=\frac{1}{2}\left(0,0,1-\frac{\sqrt{2}}{2} \right)$. One can check that $S=\{A_i,i=1,2,3,4\}$
is a maximal 1-equilateral set with center $K$ and $\mu=||K-A_i||=1-\frac{\sqrt{2}}{4},i=1,2,3,4$.
Setting $r=\mu-\frac{1}{2}=\frac{1}{2}-\frac{\sqrt{2}}{4}$, it is easy to see that the ball $B(K,r)$
touches each one of the balls $B(A_i,\frac{1}{2}),i=1,2,3,4$ (which of course touch each other).
\item Let $A_1=\left(\frac{1}{2},0,0\right)$, $A_2=\left(-\frac{1}{2},0,0\right)$ and 
$A_3=\left(0,\frac{\sqrt{3}}{2},0\right)$. The triangle $A_1A_2A_3$ of the (euclidean) $xy$-plane
is 1-equilateral with barycenter $K=\left(0,\frac{\sqrt{3}}{6},0\right)$. The three vertices of the triangle
can be extended to a four-point 1-equilateral set, either by $D=\left(0,\frac{\sqrt{3}}{6},1-\frac{\sqrt{3}}{3}\right)$
or by $D'=-D$. It is easy to check that any point equidistant from the three vertices $A_1,A_2,A_3$ must belong
to the line $DK$. Hence the set $S=\{D,A_i,i=1,2,3\}$ is maximal 1-equilateral with no center.
\end{enumerate}
\end{examples}

\begin{theo}
Let $A=\{\alpha_i:i=1,2,\dots,5\} \subseteq X$ be a 1-equilateral set. We suppose that 
$\alpha_5=(0,0,0)$, $\alpha_4=(0,d,1-d)$, where $0<d<\frac{1}{4}$ and also that the 
$z$-coordinates satisfy $\alpha'_5 \le \alpha'_3 \le \alpha'_2 \le \alpha'_1 \le \alpha'_4$.
The set $A$ has no center, i.e. there is no point $\alpha_k \in X$ such that 
$||\alpha_k-\alpha_i||=r>0, i=1,2,\dots,5$.
\end{theo}

Before proceeding with the proof of the Theorem we need two Lemmas which are easy to prove.

\begin{figure}[htbp]
\centering
\includegraphics[width=6.5cm, height=5cm]{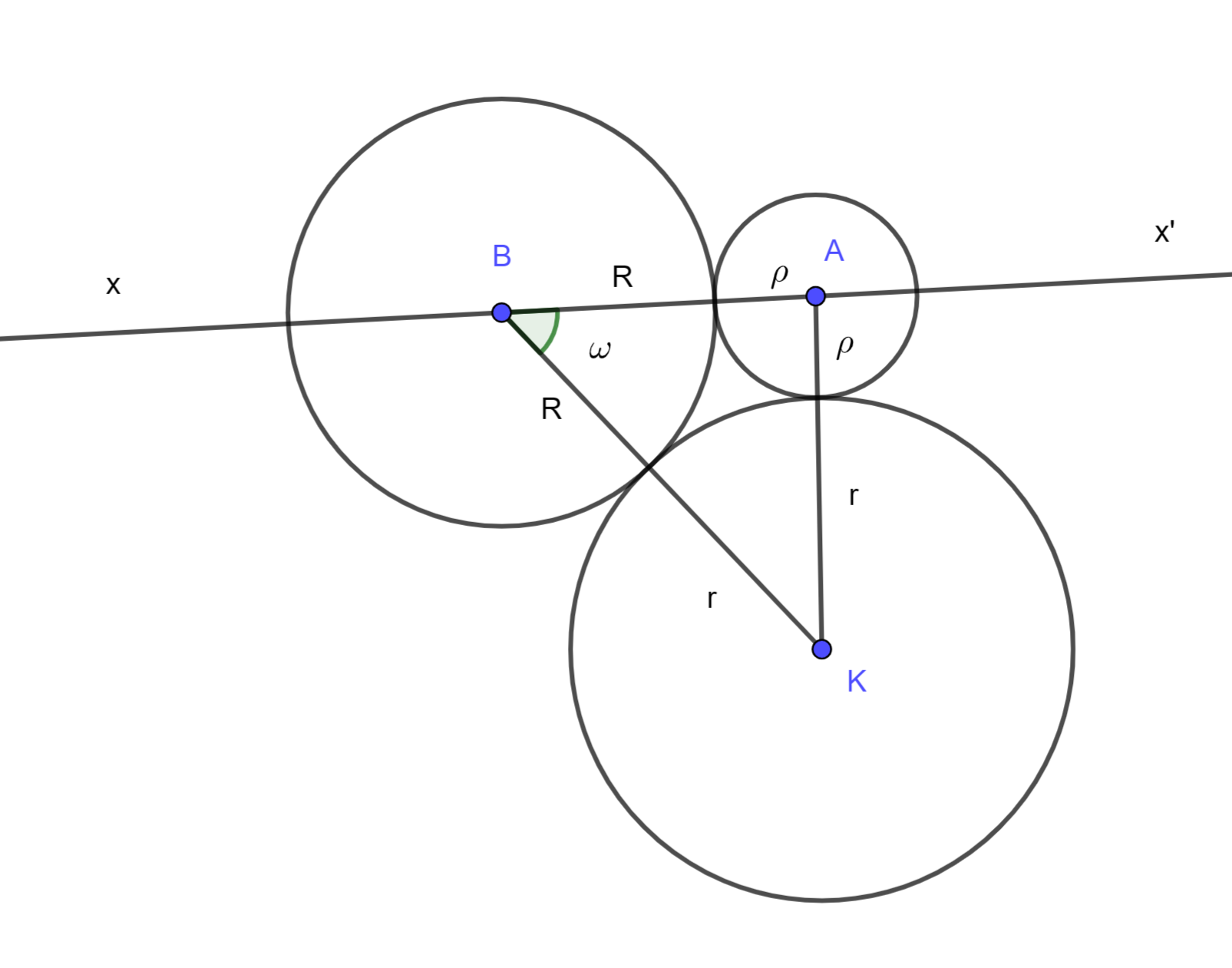}
\caption{}
\label{fig:figure 3}
\end{figure}

\begin{lemm}
Consider two tangent discs $(A,\rho)$ and $(B,R)$ on the plane with $R \ge \rho$. Let $(K,r)$
be another disc tangent to $(A,\rho)$, which may or may not touch $(B,R)$. Then $(K,r)$ has no
point in common with the half-line $ABx$.
\end{lemm}

\begin{proof}
Suppose that $(K,r)$ touches both of the other discs (the remaining case is similar).
The angle $\omega=ABK$ is acute, so the distance from $K$ to the line $xx'$ is achieved
at a point of the half-line $Bx'$. In any case the point $K$ is at distance $>r$ from
any point of the half-line $ABx$ (see Fig.3).
\end{proof}

\begin{lemm}
Consider a disc $(A,\rho)$ on the plane touching two other discs $(B,R_1)$ and $(C,R_2)$
with disjoint interiors and such that $R_1,R_2 \ge \rho$. Let $(K,r)$ be another disc tangent
to $(A,\rho)$, which may or may not touch the other two discs. Then $(K,r)$ is either contained
in the interior of the convex angle $BAC$ or in the interior of the concave angle $BAC$.
\end{lemm}

\begin{figure}[htbp]
\centering
\includegraphics[width=6.5cm, height=4.2cm]{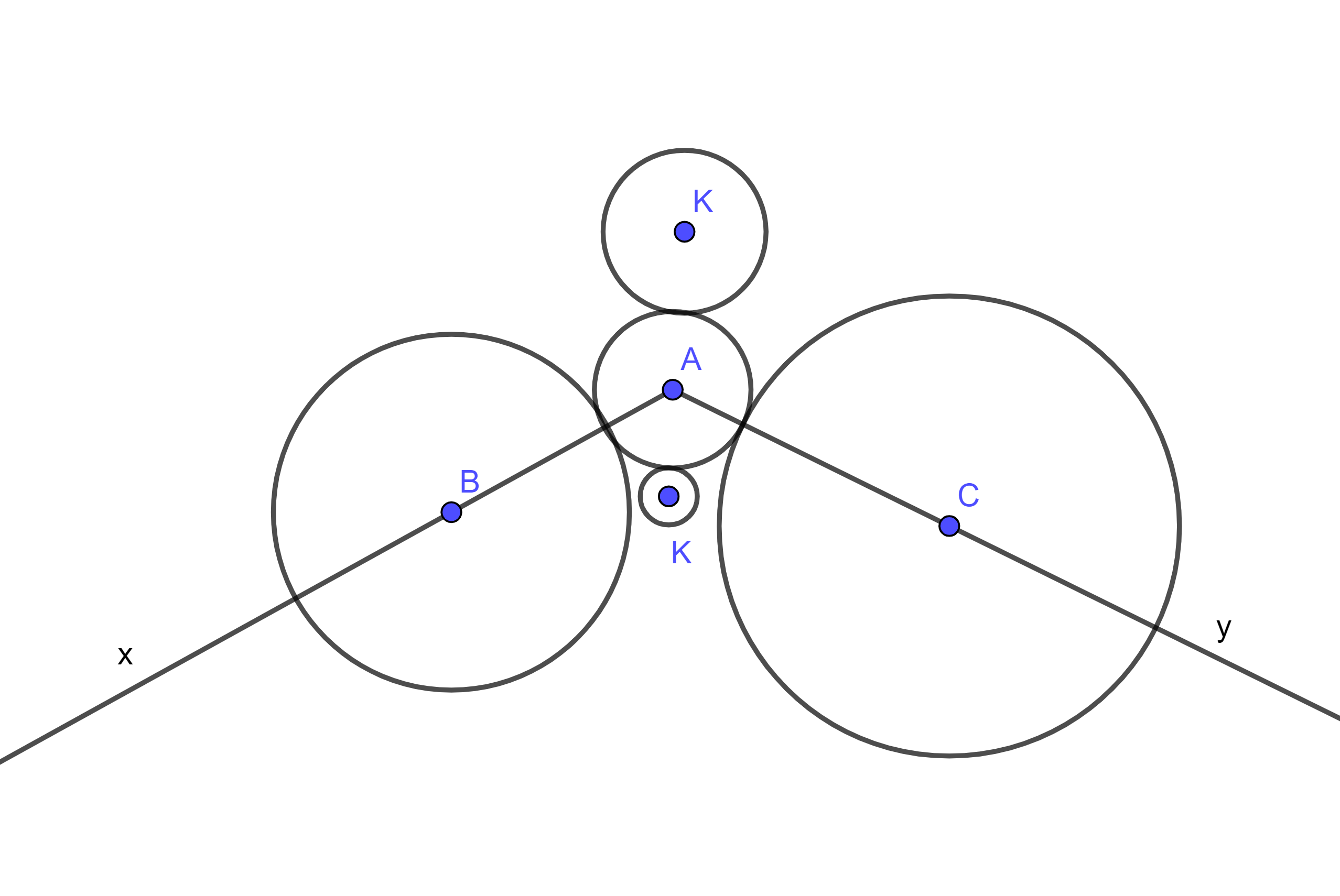}
\caption{}
\label{fig:figure 4}
\end{figure}

\begin{proof}
The proof is immediate applying Lemma 6 (see Fig.4).
\end{proof}

We also need the following

\begin{rem}
Consider two double cones $B(\alpha_4,\frac{1}{2})$ and $B(\alpha_5,\frac{1}{2})$
with $||\alpha_4-\alpha_5||=1$ and let $\alpha'_4$ and $\alpha'_5$ be the $z$-coordinates of the centers.

Suppose first that $\alpha'_4 \ge \alpha'_5$ with $\alpha'_4-\alpha'_5<\frac{1}{2}$ (see the left case of Fig.1).
The two cones touch along a segment (or have just one  point in common when $\alpha'_4=\alpha'_5$).
Intersecting the two cones with a plane $z=z_0$, $\alpha'_5 \le z_0 \le \alpha'_4$ gives two tangent discs,
while intersecting with a plane $z=z_0$, $\alpha'_4 -\frac{1}{2} \le z_0 <\alpha'_5$ gives two non-tangent discs
(or a disc and a point out of the disc).

On the other hand, in case when $\alpha'_4 \ge \alpha'_5$ with $\frac{1}{2} \le \alpha'_4-\alpha'_5<1$ (the situation 
is similar to the right case of Fig.1), intersecting with  a plane $z=z_0$, $\alpha'_4 -\frac{1}{2} \le z_0 \le \alpha'_5+ \frac{1}{2}$ gives two tangent discs (or a disc and a point on the boundary of that disc).
\end{rem}

\begin{proof} (of Theorem 1).
To obtain a contradiction, suppose that the 1-equilateral set $A$ has a center $\alpha_k \in X$ with 
$||\alpha_k-\alpha_i||=\mu>\frac{1}{2}, i=1,2,\dots,5$ and let $\alpha'_k$ be the $z$-coordinate of $\alpha_k$.
Equivalently there is a double cone with center $\alpha_k$ and radius $r=\mu-\frac{1}{2}>0$ touching each double cone 
ball $B(\alpha_i,\frac{1}{2}),i=1,2,\dots,5$.

The $z$-coordinate $\alpha'_k$ must satisfy the relation $\frac{1}{2}-d \le \alpha'_k \le \frac{1}{2}$,
since the cone $B(\alpha_k,r)$ touches both of the cones $B(\alpha_4,\frac{1}{2})$ and $B(\alpha_5,\frac{1}{2})$
(see also Prop.5).

When $\alpha'_k=\frac{1}{2}$ or $\alpha'_k=\frac{1}{2}-d$ (then $\alpha_k$ lies on the $yz$-plane) it is easy to check that
$B(\alpha_k,r)$ can only touch two out of the three balls $B(\alpha_i,\frac{1}{2}),i=1,2,3$ and due to symmetry these two balls
have centers with the same $z$-coordinate.

So we may suppose that 
\begin{equation}
\frac{1}{2}-d<\alpha'_k<\frac{1}{2}.
\end{equation}

Consider the plane $z=\alpha'_k$ whose intersection with each ball $B(\alpha_i,\frac{1}{2})$
is a closed disc $(A_i,\rho_i), i=1,2,\dots,5$ and its intersection with $B(\alpha_k,r)$ is the 
closed disc $(\alpha_k,r)$. Note that the discs $(A_4,\rho_4)$ and $(A_5,\rho_5)$ are tangent.
Also $(\alpha_k,r)$ touches each disc $(A_i,\rho_i),i=1,2,\dots,5$. Because of relation (10) 
and the fact that $d<\frac{1}{4}$ (see Prop.5), it is easy to check that $\rho_4,\rho_5$ 
are smaller than the radii $\rho_1,\rho_2,\rho_3$ (see also Rem.2).
In sum, on the plane $z=\alpha'_k$ we find two small discs $(A_4,\rho_4)$ and $(A_5,\rho_5)$
touching each other, three larger discs $(A_i,\rho_i),i=1,2,3$ and the disc $(\alpha_k,r)$
which touches any other disc. 

Regarding the arrangement of the discs on the plane there are two essential cases, depending
on where $\alpha'_k$ lies with respect to the $\alpha'_i,i=1,2,3$. The reader should take into
account Remark 2.

\begin{figure}[htbp]
\centering
\includegraphics[width=6cm, height=4.5cm]{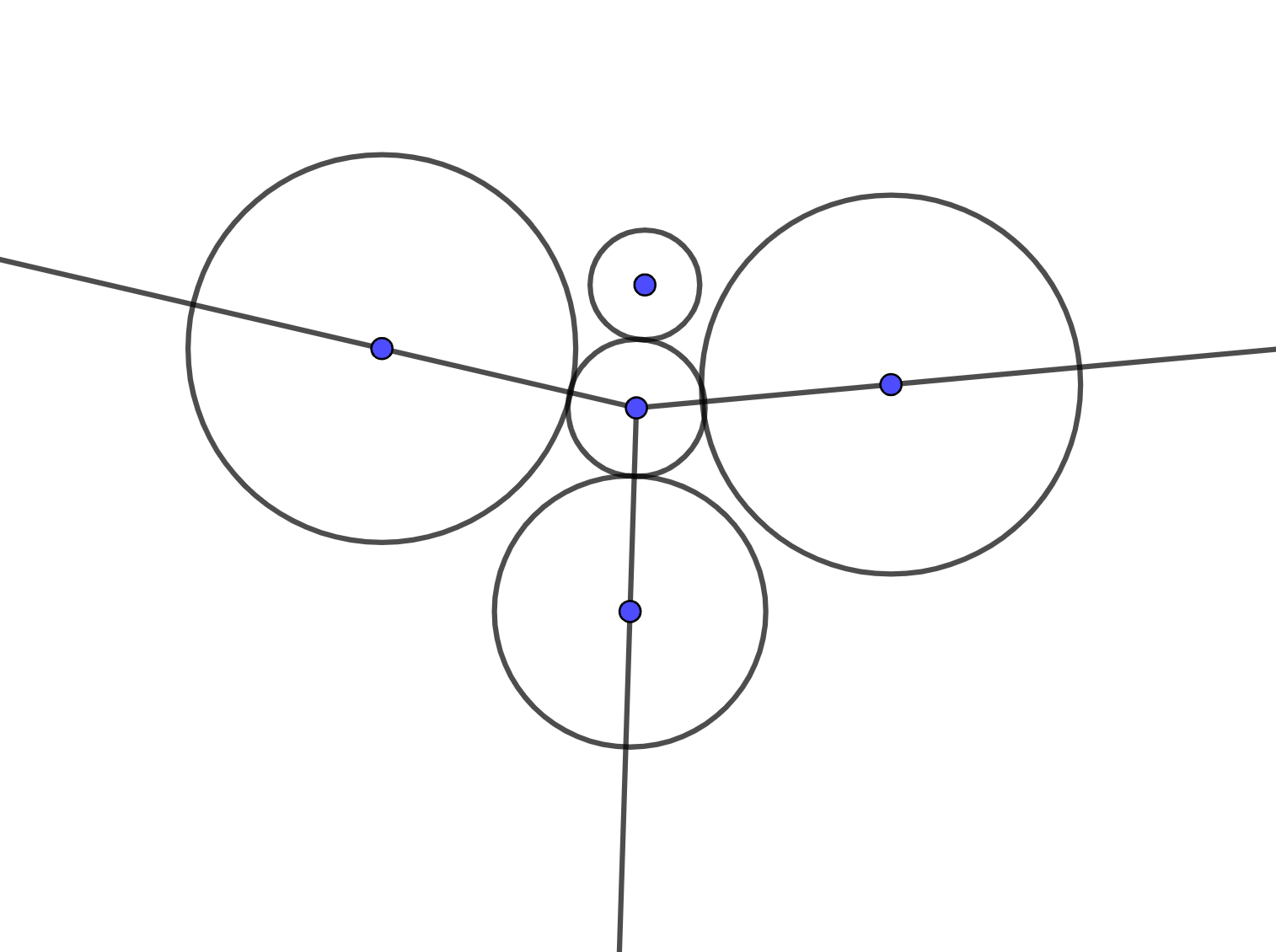}
\caption{}
\label{fig:figure 5}
\end{figure}

\noindent \textbf{Case 1:} $\alpha'_1<\alpha'_k<\frac{1}{2}$ or $\frac{1}{2}-d<\alpha'_k<\alpha'_3$.
On the plane $z=\alpha'_k$ we see the three pairwise disjoint discs $(A_i,\rho_i), i=1,2,3$ touching
the same one of the small discs. The two small discs also touch (see Fig.5).

We can now apply Lemma 7. The plane is divided into 3 regions, which are
the interiors of 3 angles generated by a small disc touching two larger discs.
The disc $(\alpha_k,r)$ must lie entirely in one of these regions. In any case 
it is impossible that $(\alpha_k,r)$ touches all of the other discs.

\noindent \textbf{Case 2:} $\alpha'_3<\alpha'_k<\alpha'_1$ and $\alpha'_k \neq \alpha'_2$.

On the plane $z=\alpha'_k$ we see two of the discs $(A_i,\rho_i), i=1,2,3$ touching
the third one and also touching the same one of the small discs. The third disc touches the other
one of the small discs and the two small discs also touch (see Fig.6).

For instance, we examine the (special) case when $\alpha'_3<\alpha'_k<\alpha'_2$ (the other case is similar). 
The plane $z=\alpha'_k$ intersects the upper half of the ball $B(\alpha_3,\frac{1}{2})$ and the lower half of the balls 
$B(\alpha_1,\frac{1}{2})$ and $B(\alpha_2,\frac{1}{2})$. The disc $(A_3,\rho_3)$ touches both discs $(A_1,\rho_1)$
and $(A_2,\rho_2)$ which are disjoint. The upper half of the ball $B(\alpha_3,\frac{1}{2})$ touches the lower half 
of $B(\alpha_4,\frac{1}{2})$, hence the disc $(A_3,\rho_3)$ touches the disc $(A_4,\rho_4)$.
On the other hand, the lower half of $B(\alpha_1,\frac{1}{2})$ and of $B(\alpha_2,\frac{1}{2})$ touch the upper
half of $B(\alpha_5,\frac{1}{2})$, hence the discs $(A_1,\rho_1)$ and $(A_2,\rho_2)$ both touch $(A_5,\rho_5)$.
Clearly, the discs $(A_4,\rho_4)$ and $(A_5,\rho_5)$ are also tangent.

\begin{figure}[bhtp]
\centering
\includegraphics[width=9.2cm, height=5cm]{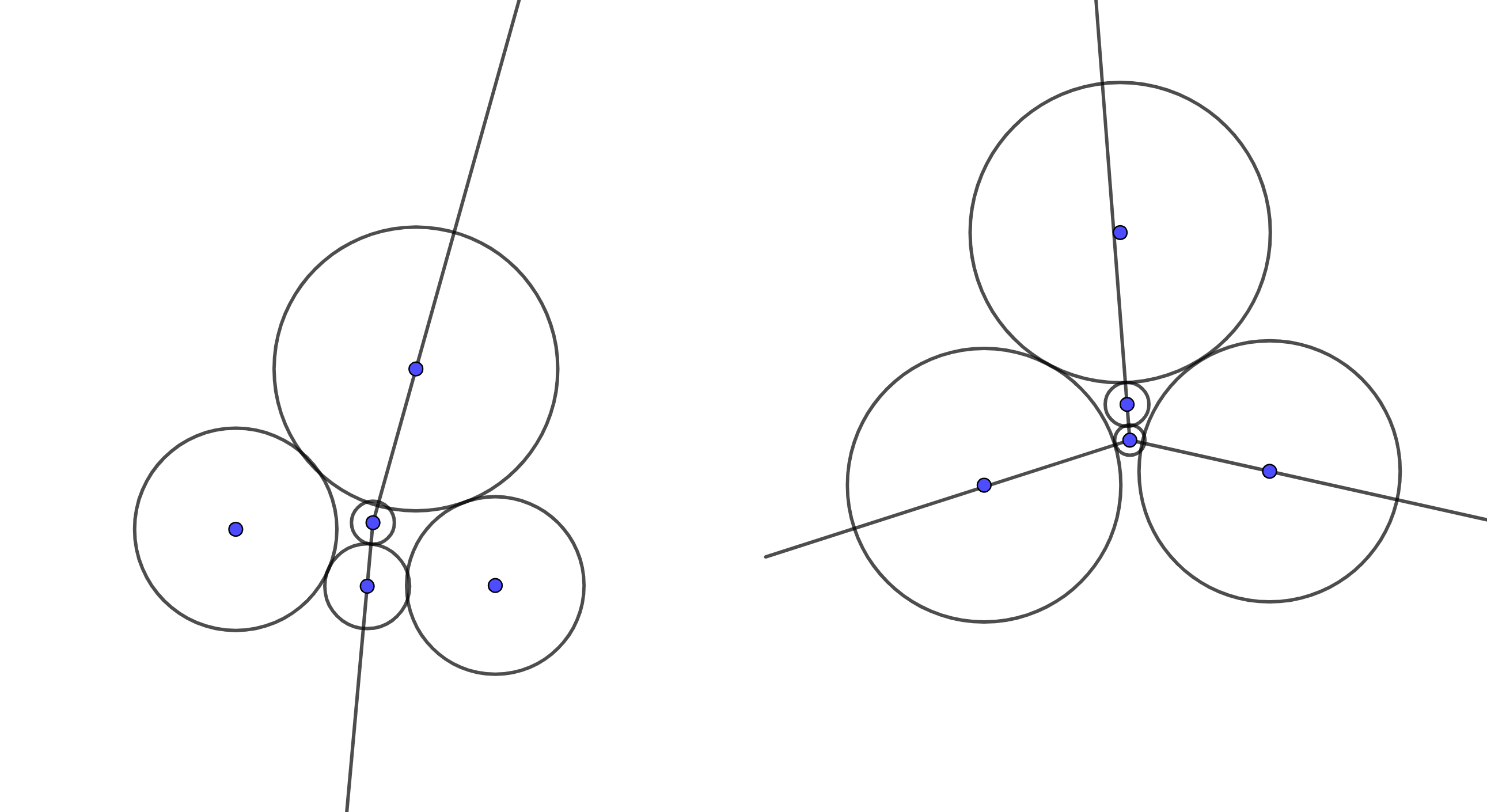}
\caption{}
\label{fig:figure 6}
\end{figure}

No matter which one of the small discs has larger radius with respect to the other, we can always
divide the plane into regions which are the interiors of angles generated by a small disc touching
two larger discs and apply Lemma 7. The disc $(\alpha_k,r)$ touches all of the other discs, yet
must lie entirely in exactly one of these regions, so this is impossible.

The remaining cases are $\alpha'_k=\alpha'_i$ for some $i=1,2,3$. Note that at most two out of
the three  $\alpha'_i,i=1,2,3$ can be equal, because of the arrangement of these points
on the corresponding ellipse (see Prop.5). When $\alpha'_k=\alpha'_1$ or $\alpha'_k=\alpha'_3$,
intersecting with $z=\alpha'_k$ we find an image similar to the one of Fig.5, except from the fact that
the disc $(A_1,\frac{1}{2})$ (or the disc $(A_3,\frac{1}{2})$) now touches all the other discs (see also Rem.2).

When $\alpha'_k=\alpha'_2$, intersecting with $z=\alpha'_k$ we find an image similar to the one of Fig.6, 
except from the fact that the disc $(A_2,\frac{1}{2})$ now touches all the other discs (see also Rem.2).

In case when $\alpha'_k=\alpha'_1=\alpha'_2$, intersecting with $z=\alpha'_k$ we find that
both of the discs $(A_1,\frac{1}{2})$ and $(A_2,\frac{1}{2})$ touche all the other discs and the case
$\alpha'_k=\alpha'_2=\alpha'_3$ is similar.

In any case, applying Lemma 7, we conclude that it is impossible to place a disc $(\alpha_k,r)$ so as to touche
each of the other disks, hence it is impossible for the set $A$ to have a center.

\end{proof}

\begin{rem}
As part of our study we investigated the configurations of 5-point equilateral sets in $X$.
The general form of such a set is the one we stated in Proposition 5 where $\alpha_5=(0,0,0)$,
$\alpha_4=(0,d,1-d)$, $d<\frac{1}{4}$, and the 3 remaining points lie on an ellipse (see Figure 2) and form an 
equilateral triangle. A natural question is which values of $d$ make it possible for an equilateral triangle on the 
corresponding ellipse to exist (and hence yield a 5-point equilateral set).

There are two extreme cases. The first one is when a vertice of the equilateral triangle coincides with one of the vertices
$A_1$ or $A_3$ of the corresponding ellipse (as in Figure 2) and the other two vertices of the triangle have the same $z$-coordinate
(due to symmetry). It is easy to check that in this case $d=d_1=\frac{\sqrt{2}-1}{4}$.

In the second case a vertice of the equilateral triangle coincides with one of the vertices
$A_2$ or $A_4$ of the ellipse (as in Figure 2) and in this case a calculation shows that $d$ is the single root
of the polynomial equation $16d^3+24d^2+8d-1=0$, 
i.e. $d=d_2=- \frac{1}{2}+\frac{1}{12} \{\sqrt[3]{54-6 \sqrt{33}}+\sqrt[3]{54+6 \sqrt{33}}\} \approx 0.09574$.

For any $d \in (d_2,d_1)$, an intermediate value argument ensures the existence of an equilateral triangle on the 
corresponding ellipse (in fact there are exactly four such triangles due to symmetry). Fix $d \in (d_2,d_1)$; then 
the isosceles triangle with vertex $A_1$ has base length $>1$, while the isosceles triangle with vertex $A_2$ has 
base length $<1$. As the vertex of an isosceles triangle moves on the arc $A_1 A_2$, the base length changes continuously,
so there is a point between $A_1$ and $A_2$ where the triangle becomes equilateral.

One can show that the values of $d$ for which there is an equilateral triangle on the ellipse (and hence a 
5-point equilateral set) are exactly those when $d \in [d_2,d_1]$ but our proof is too long to expose in this paper.
\end{rem}


\scriptsize

\noindent S.K.Mercourakis, G.Vassiliadis\\
University of Athens\\
Department of Mathematics\\
15784 Athens, Greece\\
e-mail: smercour@math.uoa.gr\\
ORCID ID 0000-0001-9867-1889\\

\noindent georgevassil@hotmail.com

\end{document}